\documentclass[10pt]{amsart}
\usepackage{amssymb,amsbsy,amsmath,amsfonts,times, amscd}
\usepackage{latexsym,euscript,exscale}
\usepackage{graphicx,color} \usepackage{amsthm}
\usepackage{fancyhdr} \usepackage{url}
\usepackage{epigraph}\usepackage{lmodern}
\usepackage{mathrsfs}\usepackage{cancel}
\usepackage[toc,page]{appendix}\usepackage[T1]{fontenc}
\usepackage[all,cmtip]{xy}\usepackage{mathtools}
\usepackage{enumerate}\usepackage{setspace}
\usepackage{helvet}
\usepackage{float}

\numberwithin{equation}{section}
\newtheorem{theorem}{Theorem}
\newtheorem{thm}[theorem]{Theorem}
\newtheorem{cor}[theorem]{Corollary}
\newtheorem{lemma}[theorem]{Lemma}

\theoremstyle{definition}

\newtheorem{problem}[theorem]{Problem}
\theoremstyle{remark}

\newcommand {\D}{\mathbb D}
\newcommand {\A}{\mathbb A}
\newcommand {\F}{\mathbb F}

\newcommand {\B}{\mathbb B}

\newcommand{\eps}{\varepsilon}
\newcommand{\E}{\mathbb{E}}

\newcommand{\cF}{\mathcal{F}}

\newcommand{\cK}{\mathcal{K}}

\newcommand{\cS}{\mathcal{S}}

\newcommand{\SB}{\text{SB}}
\newcommand{\SD}{\text{SD}}

\newcommand{\Q}{\mathbb{Q}}
\newcommand{\N}{\mathbb{N}}
\newcommand{\R}{\mathbb{R}}

\begin{document}

\title[Duality and  Zippin's theorem.]{Duality on Banach spaces and a Borel parametrized version of Zippin's theorem.}
\author{B. M. Braga }
\address{Department of Mathematics, Statistics, and Computer Science (M/C 249)\\
University of Illinois at Chicago\\
851 S. Morgan St.\\
Chicago, IL 60607-7045\\
USA}\email{demendoncabraga@gmail.com}
\date{}
\maketitle

\begin{abstract}
 Let $\SB$ be the standard coding for separable Banach spaces as subspaces of $C(\Delta)$. In these notes, we show that if $\B\subset\SB$ is a Borel subset of spaces with separable dual, then the assignment $X\mapsto X^*$ can be realized  by a \emph{Borel} function $\B\to\SB$. Moreover, this assignment can be done in such a way that the functional evaluation is still well defined (Theorem \ref{f}). Also, we prove a Borel parametrized version of Zippin's theorem, i.e., we prove that there exists $Z\in\SB$ and a \emph{Borel} function that assigns for each $X\in\B$ an isomorphic copy of $X$ inside of $Z$ (Theorem \ref{fim}).
\end{abstract}

\section{Introduction.}

 These notes mainly deal with two problems, namely, (i) how to obtain the assignment $X\mapsto X^*$ in a Borel fashion, and (ii) how to obtain a Borel parametrized version of M. Zippin's theorem. More precisely, for the duality problem,  the dual of each $X \in\SD=\{X\in\SB\ |\ X^*\text{ is separable}\}$ has an isometric copy in $\SB$.  In these notes, we show that the assignment $X\mapsto X^*$ can be obtained  
 by a \emph{Borel} function.

Recall that $\SD=\{X\in\SB\ |\ X^*\text{ is separable}\}$ is  complete coanalytic (hence non Borel). Indeed,  there is a Borel map $\Theta:\cK([0,1])\to \SB$ such that $\Theta(K)\cong C(K)$, for all $K\in\cK([0,1])$, where $X\cong Y$ means that $X$ is isomorphic to $Y$ (see \cite{Ke}, Theorem $33.24$). Therefore, as $C(K)$ is an $\ell_1$-predual if $K$ is countable, and $C(K)$ is universal for the class of separable Banach spaces if $K$ is uncountable, this gives us a Borel reduction of $\{K\in\cK([0,1])\mid K\text{ is countable}\}$ to $\SD$. As $\{K\in\cK([0,1])\mid K\text{ is countable}\}$ is complete coanalytic (see \cite{Ke}, Theorem $27.5$), $\SD$ is $\Pi^1_1$-hard, i.e., every coanalytic set Borel reduces to $\SD$. For a proof that $\SD$ is coanalytic and a detailed proof of the arguments above see \cite{Ke}, Theorem $33.24$.


As $\SD$ is non Borel,  we have to restrict ourselves to Borel subsets  of $\SD$ in order to define a Borel function. For a Borel $\mathbb{B}\subset \SD$, we show that there exists a Borel map $X\in\mathbb{B}\mapsto X^\bullet\in\SB$ such that, for all $X\in \B$, we have $X^*\equiv X^\bullet$, where $X\equiv Y$ means $X$ is isometric to $Y$. Moreover, we show that there exists a Borel  map

$$(X,x,g)\in\A\mapsto \langle g,x\rangle_X\in\R,$$\hfill

\noindent where $\A=\{(X,x,g)\in \B\times C(\Delta)\times C(\Delta)\ |\ x\in X,g\in X^\bullet\}$, that works as the functional evaluation. Precisely, we prove:

\begin{theorem}\label{f}
Let $\mathbb{B}\subset \SD$ be Borel. There exists a Borel map $\mathbb{B}\to \SB$, $X\mapsto X^\bullet$, such that  $ X^\bullet\equiv X^*$, for all $X\in \mathbb{B}$.  Moreover, let

$$\A=\{(X,x,g)\in \B\times C(\Delta)\times C(\Delta)\ |\  x\in X,g\in X^\bullet\}.$$\hfill

Then there exists a Borel map  $\langle\cdot,\cdot\rangle_{(\cdot)}:\A\to \R$ such that, for each $X\in \mathbb{B}$, 

\begin{enumerate}[(i)]
\item $\langle \cdot,\cdot\rangle_X$ is bilinear and norm continuous, and
\item  $g\in X^\bullet\mapsto \langle g,\cdot\rangle_X\in X^*$ is a surjective linear isometry.
\end{enumerate}
\end{theorem}

This result is related and can be seen as an extension of the following theorem due P. Dodos (see \cite{D2}). 

\begin{thm}\label{princesa}\textbf{(Dodos, 2010)}
Say $\SD=\{X\in\SB\ |\ X^*\text{ is separable}\}$, and let $\A\subset \SD$ be analytic. Let $\A^*=\{X\in \SB\ |\ \exists Y\in \A, Y^*\cong X\}$. Then $\A^*$ is analytic.
\end{thm}

As $\SD$ is coanalytic, if $\A\subset \SD$ is analytic, Lusin's separation theorem  says that there exists a Borel set $\B\subset \SD$ with $\A\subset \B$. Apply \text{Theorem \ref{f}}  to this $\B$, and notice that $\A^*=\{X\in \SB\mid \exists Y\in \A,\  Y^\bullet\cong X\}$. Therefore, as the isomorphism relation $\cong\ \subset\SB\times \SB$ is analytic, \text{Theorem \ref{princesa}} can be obtained from \text{Theorem \ref{f}}.

In order to prove \text{Theorem \ref{f}}, we proceed as follows. Fix a Borel $\B\subset \SD$. First, for $X\in\B$, we code the  unit ball of the dual of $X$ by a subset of $B_{\ell_\infty}$ (see \text{Lemma \ref{dodos}}). We refer to this coding as \emph{Dodos' coding} (the reader can read more about it in \cite{D2}).  Using the main technical result of \cite{D2} (for a precise statement, see \text{Lemma \ref{dodoss}} below), we code the unit ball of the bidual of $X$ as a subset of $B_{\ell_\infty}$, for all $X\in\B$ (see \text{Lemma \ref{dodosdual}}). Those codings will allow us to talk about elements of the abstract spaces $X^*$ and $X^{**}$ as elements of their concrete codings in $B_{\ell_\infty}$. This will allow us to talk about Borel functions coding the functional operations given by elements of $X^*$, and $X^{**}$.  At last, we will use those codings and \text{Lemma \ref{rosendal}} in order to bring the codings of $X^*$  inside of $\SB$. Those three steps will give us \text{Theorem \ref{f}}.

Also, while proving Theorem $1$, we obtain a coding for the functional evaluation on the entire $\SB$. It is clearly not possible to obtain an assignment $X\in\SB\mapsto X^\bullet\in\SB$  as before. Indeed,  $\SB$ contains many spaces whose duals are non separable Banach spaces, hence if we demand $X^*\equiv X^\bullet$, we cannot have $X^\bullet\in\SB$. We are however capable of coding the functional evaluation on the entire $\SB$.

\begin{thm}\label{ahaiohoi}
There exists a Borel map $\Theta:\SB\to \cK(B_{\ell_\infty})$ such that, for each $X\in\SB$, $B_{X^*}\equiv \Theta(X)$, where the isometry between $B_{X^*}$ and $\Theta(X)$ is the restriction of a linear isometry between $X^*$ and $\overline{\text{span}}\{\Theta(X)\}$. Moreover, setting 

$$\A=\{(X,x,x^*)\in \SB\times C(\Delta)\times B_{\ell_\infty}\ |\  x\in X,x^*\in \Theta(X)\},$$\hfill

\noindent there exists a Borel map  $\langle\cdot,\cdot\rangle_{(\cdot)}:\A\to \R$ such that, for each $X\in \mathbb{B}$, 

\begin{enumerate}[(i)]
\item $\langle \cdot,\cdot\rangle_X$ is  norm continuous, and
\item  $x^*\in \Theta(X)\mapsto \langle x^*,\cdot\rangle_X\in B_{X^*}$ is a surjective isometry.
\end{enumerate}
\end{thm}

It would be nice to get a global function such that its restriction to $\SD$ works as in \text{Theorem \ref{f}}.

\begin{problem}
Can we define the functions of \text{Theorem \ref{f}} globally? Precisely, is there a Borel assignment $X\in\SB\mapsto X^\bullet\in\SB$ such that, once restricted to $\SD$, the assignment has the same properties as in \text{Theorem \ref{f}}? What about a Borel map $\langle \cdot,\cdot\rangle_{(\cdot)}:\A\to\R$, where $\A=\{(X,x,g)\in\SB\times C(\Delta)\times C(\Delta)\mid x\in X,g\in X^\bullet\}$, such that, once restricted to $\A\cap ( \SD\times C(\Delta)\times C(\Delta))$, it has the same properties as in \text{Theorem \ref{f}}?
\end{problem}

The second half of the paper is devoted to Zippin's theorem. Zippin had shown (see \cite{Z}) that any Banach space with separable dual can be isomorphically embedded into a Banach space with a shrinking basis. We show the following Borel parametrized version of it. For each $Z\in\SB$, we let $\SB(Z)=\{X\in\SB\mid X\subset Z\}$.

\begin{theorem}\label{fim}
Say $\mathbb{B}\subset \SD$ is Borel. There exists a $Z\in \SD$, with a shrinking basis, and a Borel map $\Psi:\mathbb{B}\to \SB(Z)$ such that $X\cong \Psi(X)$, for all $X\in \mathbb{B}$. Moreover, setting $\mathbb{E}=\{(X,x)\in \mathbb{B}\times C(\Delta)\ |\ x\in X\}$, there exists a Borel map

$$\psi: \mathbb{E}\to Z$$\hfill

\noindent such that, letting $\psi_X=\psi(X,\cdot)$, we have that $\psi_X:X\to Z$  is a $10$-embedding, for all $X\in\mathbb{B}$. 
\end{theorem}
 
In \cite{DF}, Dodos and  V. Ferenczi had shown the following.

\begin{thm}\label{macaboa}\textbf{(Dodos and Ferenczi, 2007)}
Let $\A\subset \SD$ be analytic. There exists a Banach space $Z$ with a shrinking basis that contains an isomorphic copy of every $X\in\A$.
\end{thm}

Hence, Theorem \ref{fim} can be seen as an improvement of Dodos and  Ferenczi's theorem. Indeed, if $\A\subset \SD$ is analytic, then, as $\SD$ is coanalytic, Lusin's separation theorem gives us a Borel set $\B$ such that $\A\subset \B\subset \SD$. Therefore, applying Theorem \ref{fim} to $\B$, we obtain not only Theorem \ref{macaboa}, but also that its result can be obtained by a Borel function.

The proof of \text{Theorem \ref{fim}}, is divided into two parts. In \cite{DF}, Dodos and  Ferenczi, had shown, using results due B. Bossard (see \cite{B}),  that if $\A\subset \SD$ is analytic, then there exists an analytic set $\A'\subset \SB$ such that (i) every $X\in\A$ embeds into some $Y\in \A'$, and (ii) every $Y\in\A'$ has a shrinking basis. Following Bossard's work (see \cite{B}, or \cite{D1}, chapter $5$), we show that this result can be obtained by a Borel function. Precisely, we show that if $\B\subset \SD$ is Borel, then there exists a Borel function $\sigma:\B\to C(\Delta)^\N$ which, for each $X\in\B$, selects a shrinking basis whose span contains an isomorphic copy of $X$ (see \text{Theorem \ref{parte}} for a precise statement).

Finally, we show that if we have a Borel set of normalized shrinking basic sequences $U\subset S_{C(\Delta)}^\N$, we can find not only a space $Z\in \SD$  containing all those basis (as it is done in \cite{AD}), but also an assignment  

$$(x_n)\in U\mapsto X(\cong \overline{\text{span}}\{x_n\})\in \SB(Z)$$\hfill

\noindent  which is Borel. Combining those two steps we get the Borel parametrized version of Zippin's theorem.

Our main references for these notes are Dodos' book \emph{Banach spaces and descriptive set theory: Selected topics}, Dodos' paper \emph{Definability under Duality},   S. Argyros and Dodos'  paper \emph{Genericity and amalgamation of classes of Banach spaces}, 
B. Bossard's paper \emph{An ordinal version of some applications of the classical interpolation
theorem,} and
 Dodos and  Ferenczi's paper \emph{ Some strongly bounded classes of Banach spaces}. What we do in these notes is basically to show that some of the  results obtained in those papers can be actually obtained uniformly by Borel functions on Borel subsets of $\SD$.

\section{Notation.}\label{notation}
Let $\SB=\{X\subset C(\Delta)\ |\ X\text{ is closed and linear}\}$, $\SB$ will be our coding for the separable Banach spaces, where $C(\Delta)$ is the space of continuous functions on the Cantor set $\Delta$ (i.e., $2^\N$) endowed with the supremum norm. We endow $\SB$ with its Effros-Borel structure, i.e., the $\sigma$-algebra generated by 

$$\{X\in\SB\ |\ X\cap U\neq \emptyset\},$$\hfill

\noindent where $U$ varies among the open subsets of $C(\Delta)$. It is well known that $\SB$ is a standard Borel space with the Effros-Borel structure (see \cite{D1}, \text{Theorem 2.2}). We denote by $\SD$ the subset of $\SB$ consisting of Banach spaces with separable duals, $\SD$ is well known to be complete coanalytic (hence non Borel), as shown above. For every Banach space $X$, we denote the  unit ball of $X$ by $B_X$. Unless stated otherwise, we will always consider the unit ball $B_{X^{*}}$ endowed with its weak$^*$-topology. So, for all $X\in\SB$, $B_{X^*}$ is a compact metric space. We denote by $S_X$ the unit sphere of $X$, where $X$ is a Banach space.

Similarly as above, if $X$ is a Polish space, we can endow  $\cF(X)$, the set of non empty closed subsets of $X$,  with the Effros-Borel structure. Kuratowski and Ryll-Nardzewski's selection theorem gives us that, for any Polish space $X$, there exists a sequence of Borel functions $d_n:\cF(X)\to X$  such that, for all $F\in\cF(X)$,  the sequence $(d_n(F))_{n\in\N}$ is dense in $F$ (see \cite{Ke}, \text{Theorem 12.13}).  In these notes, we denote by $d_n:\cF(C[0,1])\to C(\Delta)$ the sequence above, where $X=C[0,1]$. Moreover, by taking rational linear combinations, we assume $(d_n)_{n\in\N}$ is closed under rational linear combinations. As $X\in \SB\mapsto B_X\in\cF(C(\Delta))$ is a Borel map, the maps $X\mapsto d_n(B_X)$ are also Borel, and $(d_n(B_X))_{n\in\N}$ is dense in $B_X$, for all $X\in\SB$.

Elements of $B_{X^*}$ will usually be denoted by $f$, while elements of $B_{\ell_\infty}$ will usually be denoted by $x^*$ or $x^{**}$ (depending whether this element is coding a functional from  $B_{X^*}$ or $B_{X^{**}}$). The reader should always have in mind that $x^*\in B_{\ell_\infty}$ actually denotes a bounded sequence $x^*=(x^*_n)_{n\in\N}\in B_{\ell_\infty}$.

In order to simplify notation, many times we omit the index of sequences, writing $(x_n)$ instead of $(x_n)_{n\in\N}$. We do the same for sums, i.e., we write $\sum_nx_n$ instead of $\sum_{n\in\N}x_n$ or $\sum_{n=1}^kx_n$. We hope this will not cause any confusion to the reader. 

When dealing with functionals, say $f\in X^*$ and $x\in X$, we use both $``f(x)"$ and $``\langle f,x\rangle"$ to denote the value of the functional $f$ evaluated at $x$. Also, as we will be dealing with many spaces and norms, in order to have a cleaner notation, we will usually simply write $\|x\|$ instead of $\|x\|_X$ to denote the norm of $x$ in $X$, where $x\in X$. The spaces in which the elements whose norms are being computed lie in should always be clear, and if there is room for any ambiguity we will specify the norm we are working with.

Say $X$ and $Y$ are Banach spaces. We write $X\equiv Y$ to denote that $X$ is linearly isometric to $Y$, and we write $X\cong Y$ to denote that $X$ is (linearly) isomorphic to $Y$. Also, if $X$ and $Y$ are metric spaces, we write $X\equiv Y$ to denote that $X$ and $Y$ are isometric as metric spaces. If $(x_n)$ and $(y_n)$ are two basic sequences, we write $(x_n)\sim (y_n)$ to denote that $(x_n)$ is equivalent to $(y_n)$, i.e., $x_n\mapsto y_n$ defines an isomorphism between $\overline{\text{span}}\{x_n\}$ and $\overline{\text{span}}\{y_n\}$.

Let $X$ be a metric space. We  denote by $\cK(X)$ the hyperspace of $X$, i.e., the space of all compact subsets of $X$ endowed with the Vietoris topology (which in metric spaces is equivalent to the topology generated by the Hausdorff metric), the reader can find more about the hyperspace $\cK(X)$ in \cite{Ke}, Section $4.F$.

In order to simplify notation  when working with many quantifiers in the same sentence, we will assume $``n,m\in\N"$ and $``\delta,\eps\in\Q^+"$. For example, we only write $``\exists \delta"$  instead of $``\exists \delta \in \Q^+"$. Similarly, $``\exists a_1,...,a_n"$ should be interpreted as $``\exists a_1,...,a_n\in \Q"$. The set in which we are quantifying over should always be clear.

Denote by $[\N]^{<\N}$ the set of all increasing finite tuples of natural numbers, and $[\N]^\N$ the set of all increasing sequence of natural numbers. As $[\N]^\N\subset \N^\N$ is Borel, we have that $[\N]^\N$ is a standard Borel space.  Also, if $A$ is any set, let $A^{<\N}$ denote the set of finite subsets of $A$. Given $s=(s_0,...,s_{n-1}),\allowbreak t=(t_0,...,t_{m-1})\in A^{<\N}$ we say that the length of $s$ is 
$|s|=n$, $s_{|i}=(s_0,...,s_{i-1})$, for all $i\in\{1,...,n\}$, and $s_{|0}=\{\emptyset\}$. We say that $s\preceq t$ \emph{iff} $n\leqslant m$ and $s_i=t_i$, for all $i\in\{0,...,n-1\}$, i.e., if $t$ is an 
extension of $s$. We define $s\prec t$ analogously.  

A subset $T$ of $A^{<\N}$ is called a 
\emph{tree} (on A) if $t\in T$ implies $t_{|i}\in T$, for all $i\in\{0,...,|t|\}$.  A tree $T$ is called pruned if for all $s\in T$ there exists $t\in T$ such that $s\prec t$. We denote by $[T]$ the set $\{\sigma\in\N^\N\ |\ \forall n\ \sigma_{|n}\in T\}$. A subset $I$ of a tree $T$ is called a \emph{segment} if $I$ is 
completely ordered and if $s,t\in I$ with $s\preceq t$, then $l\in I$, for all $l\in T$ such that $s\preceq l\preceq t$. Two segments $I_1,\ I_2$ are called \emph{completely incomparable} if neither 
$s\preceq t$ nor $t\preceq s$ hold, for all $s\in I_1$ and $t\in I_2$.

A \emph{B-tree} on a subset $A$ is a subset $S\subset A^{<\N}\setminus \{\emptyset\}$ such that $S=T\setminus \{\emptyset\}$, for some tree $T$ on  $A$. In other words, a B-tree is a tree without its root. All the definitions in the previous paragraph extend to B-trees.

Let $X\in \SB$, $A$ be a countable set, $T$ be a pruned B-tree on $A$, and $(x_t)_{t\in T}$ be a normalized sequence of elements of $X$ indexed by $T$. We say that $(X,A, T,(x_t)_{t\in T})$ is a \emph{Schauder tree basis} if 

\begin{enumerate}[(i)]
\item $X=\overline{\text{span}}\{x_t\ |\ t\in T\}$, and
\item for every $\sigma\in[T]$, the sequence $(x_{\sigma_{|n}})_{n\in\N}$ is a bi-monotone basic sequence.
\end{enumerate}

Let $(X,A, T,(x_t)_{t\in T})$ be a Schauder tree basis. We define the \emph{$\ell_2$-Baire sum} of $(X,A, T,\allowbreak (x_t)_{t\in T})$ as the completion of $c_{00}(T)$ endowed with the norm

\begin{align*}
\|z\|=\sup\Big\{\Big(\sum_{n=1}^k\|\sum_{s\in I_n} z_sx_s\|_X^2\Big)^{\frac{1}{2}}|\ I_1,...,I_k\text{ completely }&\text{incomparable}\\
&\text{ segments of }T\Big\},
\end{align*}\hfill

\noindent for all $z=(z_s)_{s\in T}\in c_{00}(T)$. By abuse of notation, we still denote by $x_t$ the elements of the $\ell_2$-Baire sum  corresponding to the original $x_t\in X$ (see \cite{D1}, chapter $3$, for details on Schauder tree basis and $\ell_2$-Baire sums).

Let $\varphi: T\to \N$ be a bijection such that, for all $s\preceq t\in T$, we have $\varphi(s)\leqslant \varphi(t)$. 

\begin{thm}\label{yy}\textbf{(see \cite{D1}, corollary $3.29$)} Let $(X,A, T,(x_t)_{t\in T})$ be a Schauder tree basis. Assume that for all $\sigma\in[T]$ the basic sequence $(x_{\sigma_{|n}})_{n\in\N}$ is shrinking. Then $(x_{\varphi^{-1}(n)})_n$ is a shrinking basis for the $\ell_2$-Baire sum of $(X,A, T,(x_t)_{t\in T})$.
\end{thm}

\section{Duality on Banach spaces.}

Our goal in this section is to prove \text{Theorem \ref{f}}, i.e., we will show how to obtain the assignment $X\mapsto X^*$  in a Borel fashion. Precisely, given a Borel subset of $\SD=\{X\in\SB\mid X^*\text{ is separable}\}$, say $\B\subset \SD$,  we will define a Borel function $X\in\B\mapsto X^\bullet\in\SB$ such that $X^\bullet $ is isometric to $X^*$, for all $X\in\B$. Moreover, we will keep track of the isometries between $X^\bullet$ and $X^*$ in such a way that it will be possible to actually interpret the elements of $X^\bullet$ as elements of $X^*$, i.e., we will be capable of computing $\langle g,x\rangle_X$ in a Borel manner, for all $X\in\B$, all $x\in X$, and all $g\in X^\bullet$.

In order to prove \text{Theorem \ref{f}}, our main tools will be Dodos' coding for the unit ball $B_{X^*}$, \text{Lemma \ref{dodoss}}, and  \text{Lemma} \ref{rosendal}, which will allow us to bring families of separable Banach spaces into our coding $\SB$.

We start by describing Dodos coding of $B_{X^*}$ as a subset of $B_{\ell_\infty}$. Given $X\in \SB$, we code $B_{X^*}$ by letting

$$K_{X^*}=\Big\{x^*\in B_{\ell_\infty}\ |\ \exists f\in B_{X^*}\forall n \ x^*_n=\frac{f(d_n(X))}{\|d_n(X)\|}\Big\}\subset B_{\ell_\infty},$$\hfill

\noindent where if $d_n(X)=0$, we let $x^*_n=0$ above. It is not hard to see (\cite{D2}, Section $3$) that the set $\mathbb{D}\subset \SB\times B_{\ell_\infty}$ defined by 

$$(X,x^*)\in \mathbb{D} \Leftrightarrow  x^*\in K_{X^*}$$\hfill

\noindent is Borel. Also, for a given $X\in \SB$, the natural map 

$$f\in B_{X^*}\mapsto \Big(\frac{f(d_n(X))}{\|d_n(X)\|}\Big)_n\in K_{X^*}$$ \hfill

\noindent is a surjective isometry (see \cite{D2}, Section $3$). Moreover, the isometry is ``linear'', i.e., if $f,g\in B_{X^*}$ and $\alpha f+\beta g\in B_{X^*}$, then 

$$\alpha f+\beta g\mapsto \alpha \Big(\frac{f(d_n(X))}{\|d_n(X)\|}\Big)_n +\beta \Big(\frac{g(d_n(X))}{\|d_n(X)\|}\Big)_n.$$\hfill

\noindent This isometry between the compact metric spaces $K_{X^*}$ and $B_{X^*}$ is actually the restriction of an isometry between the Banach spaces $\text{span} \{K_{X^*}\}$ and $X^*$. This observation will be used in the proof of \text{Theorem \ref{f}}.

We will need the following result (see \cite{Ke}, \text{Theorem $28.8$}).

\begin{theorem}\label{ooiiou}
Let $\B$ be a standard Borel space, $Y$ be a Polish space, and $\D\subset X\times Y$ be a Borel set, all of whose sections $\D_x=\{y\in Y\ |\ (x,y)\in \D\}$ are compact. Then the map $x\mapsto \D_x$ is Borel as a map $\B\to \cK(Y)$.
\end{theorem}

The lemma below is a simple application of the theorem above and it summarizes what we need regarding the Dodos' coding described above.

\begin{lemma}\label{dodos}
The map 

$$ X\in\SB\mapsto K_{X^*}\in\cK(B_{\ell_\infty})$$\hfill

\noindent is Borel. Moreover, for all $X\in \SB$, there exists an onto isometry $i_X:K_{X^*}\to B_{X^*}$ such that, if $f=i_X(x^*)$, then $x^*_n=f(d_n(X))/\|d_n(X)\|$, for all $n\in\N$ such that $d_n(X)\neq 0$, and $x^*_n=0 $ otherwise. The isometries $i_X$ are restrictions of linear isometries $\text{span}\{K_{X^*}\}\to X^*$.
\end{lemma}

\textbf{Remark:} The lemma above can also be obtained by \text{Lemma \ref{dodoss}}, which we will  state  and use below. However, as \text{Theorem \ref{ooiiou}} is more standard, we prefer to obtain this lemma by it.

We had already defined our coding for $X^*$, let us now define our coding for $X^{**}$. For this, we will need the following result of Dodos (see \cite{D2}, Section $1$). First we need to introduce some notation. Let $\A\subset \B\subset B_{\ell_\infty}$, we say that $\A$ is \emph{norm dense} in $\B$ if $\A$ is dense in $\B$ with respect to the norm topology of $B_{\ell_\infty} $. Similarly, we say $\A\subset B_{\ell_\infty}$ is \emph{norm separable} if it is separable with respect to the norm topology of $B_{\ell_\infty}$.

\begin{lemma}\label{dodoss}\textbf{(Dodos, 2010)} Let $\B$ be a standard Borel space, and let $\D\subset \B\times B_{\ell_\infty}$ be a Borel subset. Assume that, for each $x\in \B$, we have

\begin{enumerate}[(i)]
\item $\D_x=\{f\in B_{\ell_\infty}\ |\  (x,f)\in \D\}$ is non-empty and compact, and
\item $\D_x$ is norm separable.
\end{enumerate}

Then, there exists a sequence of Borel uniformizations of $\D$, $g_n:\B\to B_{\ell_\infty}$, such that $(g_n(x))_n$ is norm dense in $\D_x$, for each $x\in \B$. 
\end{lemma}

Say $\B\subset\SD$ is Borel, and define $\D$ as in Dodos' coding above. As $\B\subset \SD$, we have that $\D_X$ is norm separable, for all $X\in \B$. Therefore, by the lemma above, there exists a sequence of Borel functions $g_n:\B\to B_{\ell_\infty}$ such that, for each $X\in \B$, the sequence $(g_n(X))_n$ is norm dense in $K_{X^*}$. By taking rational linear combinations of $(g_n)$, we can assume that $(g_n)$ are closed under rational linear combinations. This sequence will play the same role as the sequence of Kuratowski Ryll-Nardzewski's selectors $(d_n)_n$ did in Dodos' coding for $B_{X^*}$.

We saw that, for each $X\in \B$,  $X^*$ is isometric to $\text{span}\{K_{X^*}\}$. In order to simplify notation, set $[K_{X^*}]=\text{span}\{K_{X^*}\}$. With that in mind, for each $X\in\B$, we define a coding for $X^{**}$ as

$$L_{X^{**}}=\Big\{x^{**}\in B_{\ell_\infty}\ |\ \exists f\in B_{[K_{X^*}]^*}\forall n \ x^{**}_n=\frac{f(g_n(X))}{\|g_n(X)\|_\infty}\Big\}\subset B_{\ell_\infty},$$\hfill

\noindent where if $g_n(X)=0$, we let $x^{**}_n=0$ above. It is not hard to see  that the set $\mathbb{D}'\subset \SB\times B_{\ell_\infty}$ defined by 

$$(X,x^{**})\in \mathbb{D}' \Leftrightarrow  x^{**}\in L_{X^{**}}$$\hfill

\noindent is Borel. Indeed, 

\begin{align*}
(X,x^{**})\in \mathbb{D}' \Leftrightarrow &\forall n,m,\ell\in\N\ \forall p,q\in \Q\\
&pg_n(X)+qg_m(X)=g_\ell(X)\\
&\rightarrow px^{**}_n\|g_n(X)\|_\infty+ qx^{**}_m\|g_m(X)\|_\infty=x^{**}_\ell\|g_\ell(X)\|_\infty.
\end{align*}\hfill

Also, for a given $X\in \SB$, the natural map 

$$f\in B_{[K_{X^*}]^*}\mapsto \Big(\frac{f(g_n(X))}{\|g_n(X)\|_\infty}\Big)_n\in L_{X^{**}}$$ \hfill

\noindent is a surjective isometry. Indeed, if $x^{**}(1),...,x^{**}(k)\in L_{X^{**}}$ and $f_1,...,f_k$ are the corresponding elements of $B_{[K_{X^*}]^*}$, then, for every $a_1,...,a_k\in\R$, we have 

\begin{align*}
\big\|\sum_{i=1}^ka_if_i\big\|_{[K_{X^*}]^*}&=\sup\big\{\big|\sum_{i=1}^ka_i
\frac{f_i(g_n(X))}{\|g_n(X)\|_\infty}\big|\ |\ g_n(X)\neq 0\big\}\\
&=\sup\big\{\big|\sum_{i=1}^ka_ix^{**}_n(i)\big|\ |\ n\in\N\big\}=\big\|\sum_{i=1}^k a_ix^{**}(i)\|_\infty.
\end{align*}\hfill
 
We can now apply \text{Theorem \ref{ooiiou}} and get the following lemma, which is the first step to show that $L_{X^{**}}$ can be used as a (nice) coding for $X^{**}$. 
 
\begin{lemma}\label{dodosdual}
Say $\mathbb{B}\subset \SD$ is Borel. The map 

$$X\in\mathbb{B}\mapsto L_{X^{**}}\in\cK(B_{\ell_\infty})$$\hfill

\noindent is Borel. Moreover, for all $X\in \B$, there exists an onto isometry $j_X:L_{X^{**}}\to B_{[K_{X^*}]^*}$  such that, if $f=j_X(x^{**})$, then $x^{**}_n=f(g_n(X))/\|g_n(X)\|$, for all $n\in\N$ such that $g_n(X)\neq 0$, and $x^{**}_n=0$ otherwise.
\end{lemma}

Before we show how to interpret the elements in our coding for $X^*$ and $X^{**}$, let us prove another lemma which will be crucial in our proof. Many times in these notes we will be working with families of separable Banach spaces which are not in $\SB$. Lemma \ref{rosendal} is the tool that we will use in order to bring those families back to $\SB$. 

For any given non-empty compact metric space $M$, there exists a continuous surjection $h:\Delta\to M$ (see \cite{Ke}, \text{Theorem} $4.18$). The following lemmas allow us to choose (in a Borel manner) continuous surjections $h_K:\Delta \to K$, for all $K\in \cK(M)$. Similar calculations can be found in \cite{S}, Proposition 3.8, page 14, and Theorem 2.1, page 106.

\begin{lemma} Let $M$  be a metric space and $L$ be a compact metric space. Let  $h:L\to M$ be a continuous function. Then the map 

$$K\in \cK(M)\mapsto h^{-1}(K)\in\cK(L)$$\hfill

\noindent is Borel.
\end{lemma}

\begin{proof}
Let $U\subset L$ be an open set. We only need to show that $\{K\in\cK(M)\mid  h^{-1}(K)\cap U\neq\emptyset\}$ is Borel (see \cite{D1}, proposition $1.4$).  We first prove the following claim.\\

\textbf{Claim:} Say $F\subset M$ is closed. Then  $\{K\in\cK(M)\mid K\cap F\neq\emptyset \}$ is Borel.\\

Say $d$ is the metric of $M$, and write $F=\cap_m V_m$, where each $V_m$ is open, and $d(x,F)<1/m$, for all $x\in V_m$. Notice that

$$\{K\in\cK(M)\mid  K\cap F\neq\emptyset\}=\cap_m\{K\in\cK(M)\mid  K\cap V_m\neq\emptyset\}.$$\hfill

Indeed, if $K\cap F\neq\emptyset$, it is clear that $K\cap V_m\neq\emptyset$, for all $m\in\N$. Say $K\cap V_m\neq\emptyset$, for all $m\in\N$, and pick $v_m\in K\cap V_m$, for each $m\in\N$. As $K$ is compact, by taking a subsequence, we can assume $v_m\to v$, for some $v\in K$. Also, as $d(v_m,F)<1/m$, for all $m\in \N$, we have that $v\in F$. Hence, $K\cap F\neq\emptyset$, and the claim is done.

Let us now finish the proof of the lemma. As $L$ is a metric space and $U$ is open, we can write $U=\cup_n F_n$, where $F_n$ is closed, for all $n\in\N$. Also, as $L$ is compact and $h$ is continuous, we have that $h(F_n)$ is closed, for all $n\in\N$. Hence, as we have

\begin{align*}
\{K\in\cK(M)\mid  h^{-1}(K)\cap U\neq\emptyset\}&=\{K\in\cK(M)\mid  K\cap h(U)\neq\emptyset\}\\
&=\cup_n\{K\in\cK(M)\mid  K\cap h(F_n)\neq\emptyset\},
\end{align*}\hfill

\noindent we are done.
\end{proof}

\begin{lemma} \label{rosendal} 
Let $\Delta$ be the Cantor set. There exists a Borel function 

$$Q:\cK(\Delta)\to C(\Delta, \Delta)$$ \hfill

\noindent such that, for each $K\in \cK(\Delta)$, $Q(K):\Delta\to \Delta$ is a continuous function onto $K$. Therefore, if $M$ is a compact metric space, and $h:\Delta \to M$ is a continuous surjection, we have that

$$H:K\in\cK(M)\mapsto  h\circ Q(h^{-1}(K))\in C(\Delta, M),$$\hfill

\noindent   is a Borel function and, for each $K\in \cK(M)$, $H(K):\Delta\to M$ is a continuous function onto $K$.
\end{lemma}

\begin{proof}
The second part of the lemma follows from the first part and the lemma above. Let us prove the first part. For each $s\in 2^{<\N}$, we let $\Delta_s=\{\sigma\in \Delta\mid s\preceq \sigma\}$. 

For each $K\in \cK(\Delta)$, we define $Q(K):\Delta\to \Delta$ as follows. If $\sigma\in K$, let $Q(K)(\sigma)=\sigma$. If $\sigma \not\in K$, let $n(K,\sigma)=\max\{n\in\N\ |\ \Delta_{\sigma_{|n}}\cap K\neq \emptyset\}$, and set 

$$Q(K)(\sigma)=\min\Delta_{\sigma_{|n(K,\sigma)}}\cap K,$$\hfill
 
\noindent where the minimum above is taken under the lexicographical order $\leqslant_{lex}$. It is easy to see that $Q(K)\in C(\Delta,\Delta)$. Let us show that $K\mapsto Q(K)$ is Borel. 

Say $g\in C(\Delta,\Delta)$,  $\delta>0$, and let $d_\Delta$ be the usual metric of $\Delta$. We need to show that $\{K\in\cK(\Delta)\ |\   \sup_{\sigma\in \Delta} d_\Delta(Q(K)(\sigma),g(\sigma))<\delta\}$ is Borel. Say $n\in\N$ and $\sigma\in \Delta$, then

\begin{align*}
\{K\in\cK(\Delta)\ |\ n(K,\sigma)=n\}=&\ \{K\in\cK(\Delta)\ |\ K\cap\Delta_{\sigma_{|n}}\neq\emptyset\}\\
&\cap\{K\in\cK(\Delta)\ |\ K\cap\Delta_{\sigma_{|n+1}}=\emptyset\}
\end{align*}\hfill

\noindent is Borel. Therefore, if $G\subset \Delta$ is a countable dense set, 

$$\{K\in\cK(\Delta)\ |\ \sup_{\sigma\in \Delta} d_\Delta(Q(K)(\sigma),g(\sigma))<\delta\}=W\cap P,$$\hfill

\noindent where 

\begin{align*}
W=\{K\in\cK(\Delta)\ |\ \exists \eps\forall \sigma\in G(\forall n\ n(K,\sigma)\neq n)\rightarrow d_\Delta(\sigma,g(\sigma))<\delta-\eps\}
\end{align*}\hfill

\noindent is Borel, and

\begin{align*}
P=\Big\{K\in\cK(\Delta)\ |\ \exists &\eps\forall \sigma\in G\ \Big(\exists n\ n(K,\sigma)= n\Big)\\
&\rightarrow \Big(\exists s\in 2^{<\N}(\sigma_{|n}\preceq s)\  \forall \sigma_{| n}\preceq s'<_{lex}s
 \ \forall \tilde{\sigma}\in G\cap \Delta_s \ \\
 & \ \ \ \ \ \   \Delta_s\cap K\neq \emptyset\wedge \Delta_{s'}\cap K=\emptyset
 \wedge d_\Delta(\tilde{\sigma},g(\sigma))<\delta-\eps\Big)\Big\}.
\end{align*}\hfill

\noindent So, $P$ is Borel, and we are done.
\end{proof}

We now show a couple of lemmas that will allow us to interpret $K_{X^*}$ and $L_{X^{**}}$ as $X^*$ and $X^{**}$, i.e., the lemmas will tell us how the functional evaluation will work if $x\in X$, $x^*\in K_{X^*}$, and $x^{**}\in L_{X^{**}}$.

\begin{lemma}\label{dual}
 For each $X\in\SB$, let $i_X$ be as in \text{Lemma \ref{dodos}}. Let $\A=\{(X,x,x^*)\in \SB\times C(\Delta)\times B_{\ell_\infty}\ |\  x\in X, x^*\in K_{X^*}\}$, and let $\alpha: \A\to \R$ be defined as 

$$\alpha(X,x,x^*)=\langle i_X(x^*),x\rangle,$$\hfill

\noindent for each $(X,x,x^*)\in \A$. Then, $\A$ is Borel, and $\alpha$ is  a Borel map.
\end{lemma}

\begin{proof}
As $X\mapsto K_{X^*}$ is Borel, it is clear that $\A$ is Borel. Pick $(X,x,x^*)\in \A$, and let $(n_j)\in [\N]^\N$ be such that $d_{n_j}(X)\to x$. Then $\alpha(X,x,x^*)=\lim x^*_{n_j}\|d_{n_j}(X)\|$. Indeed, as $d_{n_j}(X)\to x$, we have $\langle i_X(x^*), d_{n_j}(X)\rangle\to \langle i_X(x^*),x\rangle$. Hence, as  $\langle i_X(x^*), d_{n_j}(X)\rangle=x^*_{n_j}\|d_{n_j}(X)\|$, we have $ x^*_{n_j}\| d_{n_j}(X)\|\to \alpha(X,x,x^*)$.

To see that $\alpha$ is Borel notice that, given $a<b\in \R$, we have 

\begin{align*}
\{(X,x,&x^*)\in \A\ |\ \alpha(X,x,x^*)\in (a,b)\}=\\
&\{(X,x,x^*)\in \A\mid \exists \delta\forall\eps\exists n \|d_n(X)-x\|<\eps, x^*_n\| d_{n}(X)\|\in (a+\delta,b-\delta)\}.
\end{align*}
\end{proof}

Notice that, we have finally obtained  \text{Theorem \ref{ahaiohoi}}, which is the first ingredient for \text{Theorem \ref{f}}. Indeed, \text{Theorem \ref{ahaiohoi}} is a simple consequence of \text{Lemma \ref{dodos}} and \text{Lemma \ref{dual}}.

\begin{lemma}\label{dual2}
Say $\B\subset \SD$ is Borel, and let $j_X$ be as in \text{Lemma \ref{dodosdual}}. Let $\F=\{(X,x^*,x^{**})\in \B\times B_{\ell_\infty}\times B_{\ell_\infty}\ |\  x^*\in K_{X^*}, x^{**}\in L_{X^{**}}\}$, and let $\beta: \F\to \R$ be defined as 

$$\beta(X,x^*,x^{**})=\langle j_X(x^{**}),x^*\rangle,$$\hfill

\noindent for each $(X,x^*,x^{**})\in \F$. Then, $\F$ is Borel, and $\beta$ is  a Borel map.
\end{lemma}

\begin{proof}
As $X\mapsto K_{X^*}$, and $X\mapsto L_{X^{**}}$ are Borel, it is clear that $\F$ is Borel. If, in the proof of \text{Lemma \ref{dual}}, we substitute the sequence $(d_n)$ by the sequence $(g_n)$ given by \text{Theorem \ref{dodoss}}, and we substitute $i_X$ by $j_X$, the rest of the proof follows exactly as in the proof of \text{Lemma \ref{dual}}. 
\end{proof}

Notice that, as  $B_{\ell_\infty}$ is a non-empty compact metric space,  \text{Lemma \ref{rosendal}} gives us a Borel map $H:\cK(B_{\ell_\infty})\to C(\Delta,B_{\ell_\infty})$ such that, for all $K\in \cK(B_{\ell_\infty})$, $H(K):\Delta\to B_{\ell_\infty}$ is continuous and onto $K$. We have the following easy application of \text{Lemma \ref{rosendal}}, and \text{Lemma \ref{dual2}}.

\begin{cor}\label{previousc}
Let $\B\subset \SD$ be Borel. Let $H$ be as above, and $\beta$ as in \text{Lemma \ref{dual2}}. Set $\E=\{(X,x^*,y)\in \B\times B_{\ell_\infty}\times \Delta\ |\ x^*\in K_{X^*}\}$, and define $\gamma:\E\to \R$ as 

$$\gamma(X,x^*,y)=\beta(X,x^*,H(L_{X^{**}})(y)),$$\hfill

\noindent for each $(X,x^*,y)\in \E$. Then, $\E$ is Borel, and $\gamma$ is a Borel map.
\end{cor}

The following corollary is just a consequence of the previous lemmas.

\begin{cor}\label{ooo}
Assume we are in the same setting as in Corollary \ref{previousc}. Then, for all $X\in \B$, and for all $x^*\in K_{X^*}$, $\gamma(X,x^*,\cdot):\Delta\to \R$ is a continuous function, and

$$sup_{y\in\Delta}\gamma(X,x^*,y)=\|x^*\|_\infty=\|i_X(x^*)\|_{X^*}.$$\hfill
\end{cor}

The first equality in the lemma above follows from the fact that $H(L_{X^{**}}):\Delta\to B_{\ell_\infty}$ is a function \emph{onto} $L_{X^{**}}\equiv B_{X^{**}}$.

We are now ready to prove  \text{Theorem \ref{f}}, the duality theorem.  In the same fashion as in the usual proof that every separable Banach space $X$ embeds into $C(\Delta)$ (see \cite{Ke}, page $79$), we will now use the function $H$ to show that we can embed (in a Borel manner) the duals of all spaces of $X\in\B$ into $C(\Delta)$.

\begin{proof}\textbf{(of theorem \ref{f})} 
Let $\alpha$, $H$, and $\gamma$ be as in Lemma \ref{dual}, and Corollary \ref{previousc}. For each $X\in \B$, let 

$$X^\bullet=\{g\in C(\Delta)\ |\ \exists x^*\in K_{X^*}\exists\lambda\in \R \forall y\in \Delta\ \  g(y)=\lambda \gamma(X,x^*,y)\}.$$\hfill

 Let us show that the  assignment $X\mapsto X^\bullet$ is Borel. For this let $(g_n)$ be given by \text{Theorem \ref{dodoss}}, so $(g_n(X))_n$ is \emph{norm} dense in $K_{X^*}$, for all $X\in\B$.  Let $U(g,\delta)\subset C(\Delta)$ be the $\delta$-ball centered at $g$, i.e., 
 
 $$U(g, \delta)=\{f\in C(\Delta)\ |\ \exists \eps\forall y\in\Delta\ d_\Delta(f(y),g(y))<\delta-\eps\},$$\hfill
 
 \noindent  where $g\in C(\Delta)$, $\delta>0$, and $d_\Delta$ is the standard metric on $\Delta$. Let $G\subset\Delta$ be a countable dense set. We have

\begin{align*}
\{X\in \B\ |&\ X^\bullet\cap U(g, \delta)\neq\emptyset\}
=\\ 
&=\{X\in \B\ |\ \exists x^*\in K_{X^*}	\exists \lambda\ \exists \eps\forall y \in\Delta\ \ d_\Delta\big(\lambda \gamma(X,x^*,y),g(y)\big)<\delta-\eps\}\\
&= \{X\in \B\ |\ \exists n \exists\lambda\ \exists \eps\forall y\in G \ \ d_\Delta\big(\lambda \gamma(X,g_n(X),y),g(y)\big)<\delta-\eps\},
 \end{align*}\hfill

\noindent so  $X\mapsto X^\bullet$ is Borel.

Let us now define the desired map  $\langle\cdot,\cdot\rangle_{(\cdot)}:\A\to \R$, where $\A=\{(X,x,g)\in \SB\times C(\Delta)\times C(\Delta)\mid x\in X, g\in X^\bullet\}$. For each $(X,x,g)\in\A$, with $g=\lambda \gamma(X,x^*,\cdot)$, we let

$$\langle g,x\rangle_X=\lambda\alpha(X,x,x^*).$$\hfill

\noindent Let us show this map is well defined.\\

\textbf{Claim:} Fix $X\in \B$. Say $\lambda_1 \gamma(X,x^*(1),\cdot)=\lambda_2 \gamma(X,x^*(2),\cdot)$, where $\lambda_i\in\R$, and $x^*(i)\in K_{X^*}$, for $i\in\{1,2\}$. Then $\lambda_1x^*(1)=\lambda_2x^*(2)$. In particular, $\langle g, x\rangle_{X}$ does not depend on the representative $\lambda\gamma(X,x^*,\cdot)$ of $g$.\\

Indeed, by the definition of $\gamma$, we have

$$\lambda_1\langle j_X(H(L_{X^{**}})(\cdot)),x^*(1)\rangle=\lambda_2\langle j_X(H(L_{X^{**}})(\cdot)),x^*(2)\rangle,$$\hfill

\noindent which implies 

$$\langle j_X(H(L_{X^{**}})(y)),\lambda_1 x^*(1)-\lambda_2 x^*(2)\rangle=0,\ \ \ \forall y\in\Delta.$$\hfill

\noindent Therefore, as $H(L_{X^{**}}):\Delta\to L_{X^{**}}$ and $j_X:L_{X^{**}}\to B_{[K_{X^*}]^*}$ are surjective, we have that $\lambda_1 x^*(1)-\lambda_2 x^*(2)=0$, and the first part of the claim is done. By the definition of $\alpha$,  we have

$$\lambda \alpha(X,x,x^*)=\lambda\langle i_X(x^*),x\rangle.$$\hfill

\noindent Hence, as $i_X$ is linear, we conclude that $\langle g,x\rangle_X$ does not depend on the representative of $g$. So, $\langle \cdot,\cdot\rangle_{(\cdot)}$ is well defined. 

Let us now show that $\langle \cdot,\cdot\rangle_{(\cdot)}$ has the desired properties.\\

\textbf{Claim:} For each $X\in\B$, $\langle\cdot, \cdot\rangle_{X}$ is  bilinear.\\

Clearly, for a given $g\in X^\bullet$, the assignment $x\in X\mapsto \langle g, x\rangle_X\in \R$ is linear. Fix $x\in X$ and let $g_1=\lambda_1 \gamma(X,x^*(1),\cdot)\in X^\bullet$, $g_2=\lambda_2 \gamma(X,x^*(2),\cdot)\in X^\bullet$, and $g+h=\lambda_3 \gamma(X,x^*(3),\cdot)\in X^\bullet$. Similarly as in the previous claim, we have 

$$\lambda_1x^*(1)+\lambda_2x^*(2)=\lambda_3x^*(3).$$\hfill

\noindent Hence, as $i_X$ is linear, we have $\langle g+h,x\rangle_X=\langle g,x\rangle_X+\langle h,x\rangle_X$. Analogously, we have  $\langle \lambda g,x\rangle_X=\lambda \langle g,x\rangle_X$, for all $\lambda\in\R$, and we conclude that $g\in X^\bullet\mapsto \langle g,x\rangle_X\in\R$ is linear.

By \text{Corollary \ref{ooo}}, we have that  $g\in X^\bullet\mapsto \langle g,\cdot\rangle_X\in X^*$ is a surjective isometry. Indeed,  \text{Corollary \ref{ooo}} gives us that, if $g=\lambda \gamma(X,x^*,\cdot)$,

\begin{align*}
\|g\|_{C(\Delta)}&=\sup_{y\in\Delta}|\lambda \gamma(X,x^*,y)|=\|\lambda x^*\|_\infty=\|\lambda i_X(x^*)\|_{X^*}\\
&=\sup_{x\in B_{X}}|\lambda\langle i_X(x^*),x\rangle|=\sup_{x\in B_{X}}|\lambda\alpha(X,x,x^*)|\\
&=\sup_{x\in B_{X}}\langle g,x\rangle_X=\|\langle g,\cdot\rangle_X\|_{X^*}.
\end{align*}\hfill

\noindent Also, if $f\in X^*$, there exists $x^*\in K_{X^*}$, and $\lambda\in \R$ such that $f=\lambda i_X(x^*)$. Hence, letting $g=\lambda \gamma(X,x^*,\cdot)$, we have $\langle g,\cdot\rangle_X=f$, so $g\in X^\bullet\mapsto \langle g,\cdot\rangle_X\in X^*$ is surjective.

We also get for free that $\langle\cdot, \cdot\rangle_{X}$  is norm continuous, for each $X\in\B$. Let us show that $\langle \cdot,\cdot\rangle_{(\cdot)}:\A\to\R$ is Borel. For this, notice that the map 

\begin{align*}
(X,x,& x^*)\in \{(X,x,x^*)\in  \B\times C(\Delta)\times B_{\ell_\infty}\mid x\in X,x^*\in K_{X^*}\}\\ 
&\mapsto (X,x,\gamma(X,x^*,\cdot))\in  \{(X,x,g)\in \B\times C(\Delta)\times C(\Delta)\ |\ x\in X,g\in B_{X^\bullet}\}
\end{align*}\hfill

 \noindent is a Borel isomorphism, call the inverse of this map $J$.  As $\langle g,x\rangle_{X}$ does to depend on the representative of $g$, we have
 
 \begin{align*}
 \langle g,x\rangle_{X}=\Bigg\{\begin{array}{l}
 0,\text{\ \ if\ \ }g=0,\\
 \|g\|\alpha\Big(J\big(X,x,\frac{g}{\|g\|}\big)\Big), \text{ \ \ otherwise.}
  \end{array}
 \end{align*}\hfill
 
\noindent Therefore, the map $\langle\cdot, \cdot\rangle_{(\cdot)}$ is Borel, and we are done. 
\end{proof}

\section{A Borel Parametrized version of Zippin's Theorem} 

A famous theorem of Zippin says that, given a Banach space with  separable dual $X$, there exists a Banach space $Z$ with a shrinking basis such that $X$ embeds into $Z$ (see \cite{Z} for Zippin's original paper). Dodos and  Ferenczi had shown, using results from \cite{B}, that given an analytic subset $\A\subset \SD$, there exists a $Z\in \SD$ such that every $X\in\A$ embeds into $Z$ (see \cite{DF}). In other words, Dodos and Ferenczi proved a parametrized version of Zippin's theorem.

 We will now show that we can get something even stronger than a parametrized version of Zippin's theorem, we can get a \emph{Borel} parametrized version of it. Precisely, say $\B\subset \SD$ is Borel (notice, if $\A\subset \SD$ is analytic, then, as $\SD$ is coanalytic, Lusin's separation theorem gives us a Borel set $\B$ such that $\A\subset \B\subset \SD$), then one can find a space $Z\in\SD$ with a shrinking basis, and a   Borel function $\B\to \SB(Z)$ such that, for each $X\in\B$, the function assigns a  subspace of $Z$ isomorphic to  $X$ (see \text{Theorem }\ref{fim} for a precise statement).

\subsection{Embedding a Borel $\B\subset \SD$ into  spaces with shrinking bases.}

Dodos and Ferenczi had shown that for a given analytic set $\mathbb{A}\subset \SD$, there exists an analytic set $\A'\subset \SD$ such that (i) for every $X\in \mathbb{A}$, there exists an $Y\in\A'$ such that $X\hookrightarrow Y$, and (ii) $Y$ has a shrinking basis, for all $Y\in \A'$  (this was essentially done by using results of \cite{B}, the reader can find a complete proof in \cite{D1}, chapter $5$). In this subsection, we will show that we can actually find such $\A'$ by a Borel function.

Fix  a Borel $\mathbb{B}\subset \SD$. Bossard showed that for each $X\in \mathbb{B}$, there exists a sequence $(e^X_k)_k\in C(\Delta)^\N$ and a sequence of norms $(\|\cdot\|_{X,n})_n$ on $C(\Delta)$  such that, for each $X\in \mathbb{B}$, we have (a detailed construction of those objects can be found in \cite{D1}, chapter $5$):\\

\begin{enumerate}[(i)]
\item Each $\|\cdot\|_{X,n}$ is equivalent to the standard norm of $C(\Delta)$.
\item  Let $Z(X)=\{f\in C(\Delta)\ |\ \sum_n\|f\|^2_{X,n}<\infty\}$. Then $Z(X)$ is a Banach space under the norm $\|.\|_{Z(X)}=(\sum_n\|.\|^2_{X,n})^{1/2}$.
\item The inclusion $j_X:Z(X)\to C(\Delta)$ is continuous and $B_X\subset B_{Z(X)}$. So the inclusion  $\tau_X:X\subset C(\Delta)\to Z(X)$ is an embedding, and $\|\tau_X\|\leqslant 1$.
\item $(j_X^{-1}(e^X_k))_k$ is a shrinking bases for $Z(X)$. By abuse of notation, we will still denote this basis by $(e^X_k)_k$.
\end{enumerate}
\bigskip

Bossard proved the following lemmas (for detailed proofs see \cite{D1}, pages $85$ and $86$).

\begin{lemma}\label{1111}
The map $X\in \mathbb{B}\mapsto (e^X_k)_ k\in C(\Delta)^\N$ is Borel.
\end{lemma}

\begin{lemma}\label{2222}
For every $n\in\N$, the map $(X,f)\in \mathbb{B}\times C(\Delta)\mapsto  \|f\|_{X,n}\in \R$ is Borel.
\end{lemma}

Therefore, by $\|.\|_{Z(X)}$-normalizing $(e^X_k)_k$, we can assume: \\

\begin{enumerate}[(i)']
\setcounter{enumi}{3}
\item  $(e^X_k)_k$ is a \emph{normalized} shrinking bases for $Z(X)$.
\end{enumerate}
\bigskip

We need one more property of the objects described above. The norms $\|.\|_{X,n}$ are obtained  by letting

$$\|x\|_{X,n}=\inf \{\lambda>0 \ |\ \frac{x}{\lambda}\in 2^nW_X+2^{-n}B_X\},$$\hfill

\noindent where $W_X\subset C(\Delta)$ is a closed, bounded, and symmetric subset of $C(\Delta)$ defined in terms of $X$  (the map $X\in \B\mapsto W_X\in \cF(C(\Delta))$ is actually Borel, see \cite{D1}, page 86). Hence, $Z(X)$ is the \emph{$2$-interpolation space} of the pair $(C(\Delta),W_X)$ (see  Davis-Figiel-Johnson-Pelczynski \cite{DFJP}, for definition and basic facts about this interpolation space). It is easy to see, looking at the definition of interpolation spaces, that the inclusion $j_X:Z(X)\to C(\Delta)$ is continuous and it is bounded by $9K$, where

$$K=\max\{1,\sup\{|w|\ |\ w\in W_X\}\}.$$\hfill

\noindent By looking at the definition of $W_X$ (see \cite{D1}, page 83), one easily sees that $W_X\subset B_{C(\Delta)}$, so $K=1$. Therefore,  the norms of the inclusions $j_X$ are uniformly bounded by $9$, for all $X\in \B$. 

The conclusion of the discussion above is that we can assume:\\

\begin{enumerate}[(i)']
\setcounter{enumi}{2}
\item The inclusions $j_X:Z(X)\to C(\Delta)$ are continuous and their norms are uniformly bounded by $9$. As $B_X\subset B_{Z(X)}$, the inclusion  $\tau_X:X\subset C(\Delta)\to Z(X)$ is an embedding, and $\|\tau_X\|\leqslant 1$. Moreover, $\tau_X:X\subset C(\Delta)\to Z(X)$ is a $9$-embedding, for all $X\in \mathbb{B}$. 
\end{enumerate}
\bigskip

The reader should be aware that, by abuse of notation, if $x\in X$, we write $x$ every time we refer to $\tau_X(x)\in Z(X)$. As $\tau_X$ is an inclusion, we hope this will not cause any confusion.

Given $a=(a_1,...,a_k)\in \Q^{<\Q}$, let $a \times (e^X_j)$ stand for $\sum_{i=1}^ka_i e^X_i\in Z(X)$. As $\Q^{<\Q}$ is countable, we can fix an enumeration for its non zero elements, say $(\alpha_n)_n$. Given $X\in \mathbb{B}$, let

$$K_{Z(X)^*}=\Big\{z^*\in B_{\ell_\infty}\ |\ \exists f\in B_{Z(X)^*}\forall n\in\N\ \ z^*_n=\frac{f(\alpha_n\times (e^X_j))}{\| \alpha_n \times(e^X_j)\|_{Z(X)}}\Big\}\subset B_{\ell_\infty}.$$\hfill

Thus, $K_{Z(X)^*}$ is a coding for the unit ball $B_{Z(X)^*}$, and it is easy to check that $K_{Z(X)^*}\equiv B_{Z(X)^*}$. Indeed, this follows from the same arguments as when we proved that $L_{X^{**}}\equiv B_{[K_{X^*}]^*}$, right before \text{Lemma \ref{dodosdual}}. 

Define $\mathbb{D}\subset \mathbb{B}\times B_{\ell_\infty}$ by 

$$(X,z^*)\in \mathbb{D}\Leftrightarrow z^* \in K_{Z(X)^*}.$$\hfill 

\noindent Then $\mathbb{D}$ is Borel. Indeed, we only need to notice that

\begin{align*}
(X,z^*)\in \mathbb{D} &\Leftrightarrow \forall \alpha_n,\alpha_m,\alpha_\ell\in \Q^{<\Q} \forall p,q\in \Q\\
&p\alpha_n\times (e^X_j)+q\alpha_m\times (e^X_j)=\alpha_\ell\times (e^X_j)\\
&\rightarrow pz^*_n\|\alpha_n\times (e^X_j)\|_{Z(X)}+ qz^*_m\|\alpha_m\times (e^X_j)\|_{Z(X)}=z^*_\ell\|\alpha_\ell\times (e^X_j)\|_{Z(X)},
\end{align*}\hfill

\noindent so, by \text{Lemma \ref{1111}}, and \text{Lemma \ref{2222}}, $\mathbb{D}$ is Borel. \text{Theorem \ref{ooiiou}}
 gives us the following.
 
\begin{lemma}\label{memata}
Assume $\B\subset \SD$ is Borel. The map 

$$X\in \mathbb{B}\mapsto K_{Z(X)^*}\in \cK(B_{\ell_\infty})$$\hfill

\noindent is Borel. Moreover, for all $X\in \mathbb{B}$, there exists an onto isometry $i_X:K_{Z(X)^*}\to B_{Z(X)^*}$ such that, if $f=i_X(z^*)$, then $z^*_n= f(\alpha_n\times (e^X_j))/\| \alpha_n \times(e^X_j)\|_{Z(X)}$.
\end{lemma}

The following lemmas will play the same role \text{Lemma \ref{dual}}, \text{Lemma \ref{dual2}}, and \text{Lemma \ref{ooo}} played in the previous section.

\begin{lemma}\label{osana}
Say $\B\subset \SD$ is Borel, and let $i_X$ be as in \text{Lemma \ref{memata}}. Let $\mathbb{A}=\{(X,n,z^*)\in \mathbb{B}\times  \N\times B_{\ell_\infty} \ |\ z^*\in K_{Z(X)^*}\}$. Define $\alpha:\mathbb{A}\to \R$ as

$$\alpha(X,n,z^*)= \langle i_X(z^*),\alpha_n\times (e^X_j)\rangle ,$$\hfill

\noindent  for each $(X,n,z^*)\in\A$. Then, $\mathbb{A}$ is a Borel set, and $\alpha$ a Borel map.
\end{lemma}

Notice that, for all $(X,n,z^*)\in \mathbb{A}$, we have $\alpha(X,n,z^*)=z^*_n\|\alpha_n\times (e^X_j)\|_{Z(X)}$. The proof of this lemma is analogous to the proof of \text{Lemma \ref{dual}}. 

The inclusion $\tau_X:X\to Z(X)$ is an embedding, therefore, for each $x\in X$,  there exists a sequence $(\alpha_{n_k}\times (e^X_j))_k$ converging to $x$ in $Z(X)$.  With this observation in mind we have the following.

\begin{lemma}\label{osana2}
Let $\B\subset \SD$ be Borel, and let $i_X$ be as in \text{Lemma \ref{memata}}. Set $\mathbb{A}'=\{(X,x, z^*)\in \mathbb{B}\times  C(\Delta)\times B_{\ell_\infty} \ |\ x\in X,z^*\in K_{Z(X)^*}\}$, and define $\alpha':\mathbb{A}'\to \R$ by 

$$\alpha'(X,x,z^*)=\langle i_X(z^*),x\rangle.$$\hfill

\noindent for each $(X,x,z^*)\in \A' $. Then, $\mathbb{A}'$ is a Borel set, and $\alpha'$ a Borel map.
\end{lemma}

Notice that, when we write $``\langle i_X(z^*),x\rangle"$, we are thinking of $x$ as an element of $Z(X)$ in order for this to make sense. For each $(X,x,z^*)\in \mathbb{A}'$, we have $\alpha'(X,x,z^*)= \lim_k z^*_{n_k}\|\alpha_{n_k}\times (e^X_j)\|_{Z(X)}$, where $\alpha_{n_k}\times (e^X_j)\to x$ in $Z(X)$. Hence, the proof of this lemma is also analogous to the proof of \text{Lemma \ref{dual}}.

We now prove the main theorem of this subsection.

\begin{theorem}\label{parte}
 Let $\mathbb{B}\subset \SD$ be Borel, and let $\mathbb{E}=\{(X,x)\in \mathbb{B}\times C(\Delta)\ |\ x\in X\}$. There are Borel maps 

$$\sigma:\mathbb{B}\to C(\Delta)^\N\text{ \ \ and \ \ }\varphi:\mathbb{E}\to C(\Delta)$$\hfill

\noindent such that, by setting $\varphi_X=\varphi(X,\cdot)$, we have that, for each $X\in \B$,

\begin{enumerate}[(i)]
\item $\sigma(X)$ is a normalized shrinking basic sequence, and
\item $\text{Im}(\varphi_X)\subset  \overline{\text{span}}\{\sigma(X)\}$ and $\varphi_X: X\to \overline{\text{span}}\{\sigma(X)\}$ is a $9$-embedding.
\end{enumerate}
\end{theorem}

\begin{proof}
Let $H:\cK(B_{\ell_\infty})\to C(\Delta, B_{\ell_\infty})$ be given by \text{Lemma \ref{rosendal}}, and let $\alpha$ and $\alpha'$ be given by \text{Lemma \ref{osana}}, and \text{Lemma \ref{osana2}}, respectively. Fix  $(n_k)\in\N^\N$ such that, for each $k\in\N$, $\alpha_{n_k}\times (e^X_j)=e^X_k$ (notice that this does not depend on $X$).  For each $X\in \mathbb{B}$, let

$$\sigma(X)=\big(\alpha(X,n_k,H(K_{Z(X)^*})(\cdot))\big)_k.$$\hfill

By \text{Lemma \ref{osana}}, and as $H$ and $X\mapsto K_{Z(X)^*}$ are Borel, it is clear that $\sigma$ is Borel. Also,  $\sigma(X)$ is equivalent to the $(e^X_k)_k$, so $\sigma(X)$ is normalized and shrinking. Indeed,  we can define a map on $Z(X)$ by letting

$$\alpha_n\times(e^X_j)\in Z(X)\mapsto \alpha(X,n,H(K_{Z(X)^*})(\cdot))\in \overline{\text{span}}\{\sigma(X)\},$$\hfill

\noindent and extending it to the entire $Z(X)$, making it continuous. This map is clearly surjective. Also, as $H(K_{Z(X)^*}):\Delta\to K_{Z(X)^*}$, and $i_X:K_{Z(X)^*}\to B_{Z(X)^*}$ are surjective, we have that 

$$ \langle i_X(H(K_{Z(X)^*})(y)),x_1-x_2\rangle =0, \ \ \forall y\in \Delta, $$\hfill

\noindent can only happen if $x_1=x_2$. So, by the definition of $\alpha$, this map is a bijection. Again by the definition of $\alpha$, we have

$$\|\alpha_n\times(e^X_j)\|_{Z(X)}=\sup_{y\in\Delta}|\alpha(X,n,H(K_{Z(X)^*})(y))|,$$\hfill

\noindent hence, this map defines a surjective isometry between $Z(X)$ and $\overline{\text{span}}\{\sigma(X)\}$, and the basis $(e^X_k)$ is sent to 

$$\big(\alpha(X,n_k,H(K_{Z(X)^*})(\cdot))\big)_k=\sigma(X).$$\hfill

\noindent Therefore, $\sigma(X)$ is indeed $1$-equivalent to the $(e^X_k)_k$.

Let $\varphi(X,x)=\alpha'(X,x,H(K_{Z(X)^*})(\cdot)),$ for all $(X,x)\in\mathbb{E}$. \text{Lemma \ref{osana2}} gives us that $\varphi$ is Borel. Notice that $\varphi_X$ is the composition of  $\tau_X:X\subset C(\Delta)\to Z(X)$  with the  isometry 

$$Z(X)\to\overline{\text{span}}\{\sigma(X)\}$$\hfill

\noindent described above. Therefore, as the inclusion $\tau_X:X\subset C(\Delta)\to Z(X)$ is a $9$-embedding, we have  that $\varphi_X$ is a $9$-embedding, for all $X\in \B$. 
\end{proof}

\subsection{Embedding a Borel set of bases into  a single space with a shrinking bases.}

In \cite{AD},  Argyros and Dodos showed that if  $\mathbb{A}\subset \SD$ is analytic and $X$ has a shrinking Schauder basis, for all $X\in \mathbb{A}$, then $\mathbb{A}$ can be embedded into a single $Z\in\SD$, i.e., there exists $Z\in \SD$ such that $X\hookrightarrow Z$, for all $X\in \mathbb{A}$. In this section, we will follow Argyros and Dodos' method in order to embed $\mathbb{A}$ into a single $Z\in \SD$ in a \emph{Borel} manner.

Let $bs=\{(f_n)\in S_{C(\Delta)}^\N\ |\ (f_n) \text{ is a basic sequence}\}$. It is easy to see that the set of basic sequences $bs$ is Borel in $C(\Delta)^\N$. 

Let us recall Schechtman's  construction of Pelczynski's universal space (see \cite{Sc}). Let $(d_n)$ be a dense sequence in the unit sphere of $C(\Delta)$. For each $s=(n_1,...,n_k)\in [\N]^{<\N}$, let $m_s=n_k$. For each $s\in [\N]^{<\N}$, let $g_s=d_{m_s}$. The universal Pelczynski space $U$ is defined as the completion of $c_{00}([\N]^{<\N})$ under the norm

$$\|x\|_U=\sup\{\|\sum_{s\in I} x_s g_s\|\ |\ I \text{ is a segment of }[\N]^{<\N}\},$$\hfill

\noindent for all $x=(x_s)\in c_{00}([\N]^{ <\N})$. By taking an isometric copy of $U$ inside of $C(\Delta)$, we can assume $U\subset C(\Delta)$. Fix a bijection $\varphi: [\N]^{<\N}\to \N$ such that if $s_1\prec s_2$ then $\varphi(s_1)< \varphi(s_2)$. It is easy to see that $(g_{\varphi^{-1}(n)})_n$ is a bases for $U$. 

This construction of $U$ gives us that if $(f_k)\in S_{C(\Delta)}^\N$ is a  basic sequence and $(d_{n_k})$ is a subsequence of $(d_n)$ close enough  to $(f_n)$, then $(f_k)_k\sim (g_{(n_1, ...,n_k)})_k$. More precisely, if $(f_k)$ has basic constant $K$, $\theta\in(0,1)$, and $\|f_k-d_{n_k}\|< 2^{-k}\theta\|f_k\|/2K$, for all $k\in\N$, the principle of small perturbation gives us that (see \cite{AK}, \text{Theorem 1.3.9})

\begin{enumerate}[(i)]
\item $(f_k)_k\sim (g_{(n_1, ...,n_k)})_k$, and
\item the isomorphism  between $\overline{\text{span}}\{f_k\}\subset C(\Delta)$ and $\overline{\text{span}}\{ g_{(n_1, ...,n_k)}\}\subset U$, given by $f_k\mapsto g_{(n_1, ...,n_k)}$, is an $\frac{1+\theta}{1-\theta}$-isomorphism.
\end{enumerate}

Fix $\theta>0$. Let us define a function $b_\theta: bs\to [\N]^\N$. For this, given any basic sequence $(f_n)\in S_{C(\Delta)}^\N$ with basic constant $K$, we produce a subsequence of $(d_n)$ as follows. Say $n_1<...<n_k$ had been chosen. Let $n_{k+1}$ be the first natural number such that $n_{k+1}>n_k$ and 

$$\|f_{k+1} -d_{n_{k+1}}\|<\frac{2^{-(k+1)}}{2K}\theta\|f_{k+1}\|.$$ \hfill

The map $k:bs\mapsto \R$ that assigns to each basic sequence  its basic constant is Borel, indeed, given $b\in \R$,

\begin{align*}
(f_n)\in \{(g_n)\in bs\ |\ k((g_n)_n)<b\}\ \Leftrightarrow\ &\exists r\in (-\infty, b)\cap \Q\ \forall a_1,...,a_m\in \Q\\
&\forall k\leqslant m \ \ \| \sum_{i=1}^ka_if_i\|\leqslant r\| \sum_{i=1}^ma_if_i\|.
\end{align*}\hfill

\noindent Therefore, it should be clear that, for any fixed $\theta>0$, the function $b_\theta:bs\to [\N]^\N$ described in the previous paragraph is Borel. Also, we should keep in mind that, if $b_\theta((f_k)_k)=(n_k)_k$, the isomorphism between $\overline{\text{span}}\{f_k\}$ and $\overline{\text{span}}\{g_{(n_1,...,n_k)}\}\subset U$ is given by $f_k\mapsto g_{(n_1,...,n_k)}$. 

Let $\cS\subset [\N]^\N$ be the standard coding for shrinking basic subsequences of the bases of $U$, i.e., we define $\cS$ by

$$\cS=\{(n_k)_k\in[\N]^\N\ |\  (g_{(n_1,...,n_k)})_k\in U^\N\text{ is shrinking}\}.$$\hfill

\noindent It is known that $\cS$ is a coanalytic set (see for example \cite{D1}, Section $2.5.3$). 

We now follow Argyros and Dodos' approach. Fix $\theta\in(0,1)$. Let $\mathbb{B}\subset \SD$ be Borel, and $\sigma$ be as in \text{Theorem \ref{parte}}. Let $\xi_\theta=b_\theta\circ\sigma: \mathbb{B}\to [\N]^\N$, so $\xi_\theta(\mathbb{B})\subset [\N]^\N$ is analytic, and $\xi_\theta(\mathbb{B})\subset \cS$. As $\cS$ is coanalytic, Lusin's separation theorem gives us  a Borel set $A$ such that $\xi_\theta(\mathbb{B})\subset A\subset \cS$ (see \cite{Ke}, \text{Theorem 14.7}). Therefore, there exists a pruned tree $T$ on $\N\times\N$, such that the projection $p([T])=\{\sigma\in [\N]^\N\ |\ \exists \tau\ (\sigma,\tau)\in [T]\}$ equals $A$ (see \cite{D1}, \text{Theorem 1.6}). For each $t=(s,v)=((s_1,...,s_k),(v_1,...,v_k))\in (\N\times\N)^{<\N}$, let $x_t=g_{(s_1,...,s_k)}$. 

Let $U'=\overline{\text{span}}\{x_t\ |\ t\in T\}$, and $T'=T\setminus \{\emptyset\}$. So $(U', \N\times\N, T', (x_t)_{t\in T'})$ is a Schauder tree basis (see the last paragraphs of Section $2$ for definitions). 

Let $Z$ be the $\ell_2$-Baire sum of the Schauder tree basis $(U', \N\times\N, T', (x_t)_{t\in T'})$. Then, by \text{Theorem \ref{yy}}, $(x_t)_{t\in T}$ is  a shrinking bases for $Z$. In particular, $Z\in \SD$. 

We have the following trivial lemma.

\begin{lemma}\label{lemma25}
Let $\gamma: \N^\N\to Z^\N$ be defined by $\gamma((n_k)_k)=(g_{(n_1,...,n_k)})_k$. Then $\gamma$ is continuous, hence Borel. 
\end{lemma}

Before proving  \text{Theorem \ref{fim}} let us prove one more lemma. For each $X\in\SB$, we let $bs(X)=\{(f_n)\in S_{X}^\N\ |\ (f_n)\text{ is basic}\}$. So $bs(X)$ is Borel.

\begin{lemma}\label{lll}
Say $Y,Z\in \SB$. Let 

$$\mathbb{A}=\{((f_n),(g_n),x)\in bs(Y)\times bs(Z)\times Y\ |\  (f_n)\sim (g_n)\}.$$ \hfill

\noindent For each basic sequences $(f_n)\in bs(Y)$ and $(g_n)\in bs(Z)$, denote by $I_{f,g}$ the linear map such that $f_n\mapsto g_n$. Then, $\mathbb{A}$ is Borel and the map

$$((f_n),(g_n),x)\in \mathbb{A}\mapsto I_{f,g}(x)\in Z$$\hfill

\noindent is Borel.
\end{lemma}

\begin{proof}
Clearly, $\mathbb{A}$ is Borel. In order to see that this map is Borel, first notice that if $C=\{((f_n),(g_n)) \in bs(Y)\times bs(Z)\ |\ (f_n)\sim (g_n)\}$ then 

$$((f_n),(g_n))\in C\mapsto \|I_{f,g}\|\in \R$$\hfill

\noindent is Borel. Indeed, say $b\in \R$, then

\begin{align*}
((f_n),(g_n))\in\{((f_n),(g_n)) \in C\ |\ \|I_{f,g}\|<b\}\ \Leftrightarrow\ &\exists r\in (-\infty, b)\cap \Q\ \forall a_1,...,a_m\in \Q\\
&\ \ \| \sum_{i=1}^ma_if_i\|\leqslant r\| \sum_{i=1}^ma_ig_i\|.
\end{align*}\hfill

Let $U\subset C(\Delta)$ be an open ball open. Then

\begin{align*}
I_{f,g}(x)\in U\neq \emptyset \ \Leftrightarrow\ &\exists \delta\exists a_1,...,a_n\ \|x-\sum a_if_i\|_X<\frac{\delta}{\|I_{f,g}\|}\\
&\wedge\Big(\forall m \|d_m(Z)-\sum a_ig_i\|_Z<\delta\Big)\rightarrow d_m(Z)\in U,
\end{align*}\hfill

\noindent and we are done.
\end{proof}

\begin{proof}\textbf{(of theorem \ref{fim})} Let $\mathbb{B}\subset \SD$ be Borel. Fix $\theta\in(0,1)$. Let $Z$ be the $\ell_2$-Baire sum described in the discussion preceeding Lemma \ref{lemma25}. Let $\sigma$ and $\varphi$ be as in \text{Theorem \ref{parte}}, and $b_\theta$ and $\gamma$ be as above. Let $\chi_\theta= \gamma\circ b_\theta \circ\sigma$, and define, for each $(X,x)\in \E$,

$$\psi(X,x)= I_{\sigma(X),\chi_\theta(X)}(\varphi(X,x)).$$\hfill

Notice that, for all $X\in\B$, the sequence $\sigma(X)$ is equivalent to $\chi_\theta(X)$, so $I_{\sigma(X),\chi_\theta(X)}$ is well defined. By \text{Lemma \ref{lll}} and the fact that $\sigma$, $\varphi$ and $\chi_\theta$ are Borel, we have that $\psi$ is Borel. By \text{Theorem \ref{parte}}, we have that $\varphi_X:X\to \overline{\text{span}}\{\sigma(X)\}$ is a $9$-embedding, for all $X\in \mathbb{B}$. By the construction of $\chi_\theta$, we have that 

$$I_{\sigma(X),\chi_\theta(X)}:\overline{\text{span}}\{\sigma(X)\}\to Z$$\hfill

\noindent  is an $\frac{1+\theta}{1-\theta}$-embedding, for all $X\in\B$. Hence, by choosing $\theta$ small enough, $\psi_X$ is a $10$-embedding, for all $X\in \B$.

For each $X\in \mathbb{B}$, define $\Psi(X)=\overline{\text{span}}\{\psi_X(d_n(X))\ |\ n\in\N\}$. It is clear that $\Psi$ is Borel, and that it has the desired properties.
\end{proof}

\textbf{Acknowledgements:} The author would like to thank his adviser C. Rosendal for all the help and attention he gave to this paper.

\end{document}